\documentclass[12pt]{amsart}

\usepackage{amssymb,latexsym}
\usepackage[enableskew]{youngtab}
\usepackage{verbatim,euscript}
\usepackage{fullpage,color}
\usepackage[all]{xy}

\usepackage[colorlinks=true,linkcolor=blue,urlcolor=my_color,citecolor=magenta]{hyperref}

\definecolor{my_color}{rgb}{0,0.5,0.5}
\definecolor{MIXT}{rgb}{0.4,0.3,0.6}

\input {cyracc.def}
\numberwithin{equation}{section}

\newtheorem{thm}{Theorem}[section]
\newtheorem{lm}[thm]{Lemma}
\newtheorem{cl}[thm]{Corollary}
\newtheorem{prop}[thm]{Proposition}

\theoremstyle{remark}
\newtheorem{rmk}[thm]{Remark}

\theoremstyle{definition}
\newtheorem{ex}[thm]{Example}
\newtheorem{df}{Definition}
\newtheorem*{rema}{Remark}

\newenvironment{proof*}
{\noindent {\sl Proof.}\quad }{\hfill $\square$}

\newcommand {\ah}{{\mathfrak a}}
\newcommand {\be}{{\mathfrak b}}
\newcommand {\ce}{{\mathfrak c}}

\newcommand {\g}{{\mathfrak g}}

\newcommand {\el}{{\mathfrak l}}

\newcommand {\te}{{\mathfrak t}}
\newcommand {\ut}{{\mathfrak u}}
\newcommand {\z}{{\mathfrak z}}

\newcommand {\gH}{{\eus H}}

\newcommand {\gS}{{\eus S}}

\newcommand {\esi}{\varepsilon}
\newcommand {\ap}{\alpha}

\newcommand {\lb}{\lambda}

\newcommand {\HW}{\widehat W}
\newcommand {\HV}{\widehat V}
\newcommand {\HP}{\widehat\Pi}
\newcommand {\HD}{\widehat\Delta}

\newcommand {\BR}{{\mathbb R}}
\newcommand {\BN}{{\mathbb N}}

\newcommand {\hot}{{\mathsf{ht}}}

\newcommand {\rk}{{\mathsf{rk\,}}}
\newcommand {\rt}{{\mathsf{rt}}}

\newcommand {\supp}{{\mathsf{supp}}}

\newcommand {\GR}[2]{{\textrm{{\bf #1}}}_{#2}}

\newcommand {\un}{\underline}

\newcommand {\Ab}{\mathfrak{Ab}}
\newcommand {\Abo}{\mathfrak{Ab}^o}

\newcommand {\beq}{\begin{equation}}
\newcommand {\eeq}{\end{equation}}

\newcommand{\curge}{\succcurlyeq}
\newcommand{\curle}{\preccurlyeq}
\renewcommand{\le}{\leqslant}
\renewcommand{\ge}{\geqslant}

\newenvironment{E8}[8]{%
{\small\begin{tabular}{@{}c@{}}
{#1}-{#2}-{#3}-{#4}-\lower3.5ex\vbox{\hbox{{#5}\rule{0ex}{2.5ex}}
\hbox{\hspace{0.4ex}\rule{.2ex}{1ex}\rule{0ex}{1.4ex}}\hbox{{#8}\strut}}-{#6}-{#7}
\end{tabular}}}

\newcommand{\eus}{\EuScript}

\newfam\Bbbfam\newfam\eufam\newfam\eusfam%
\font\Bbbfont=msbm10 scaled 1200%
\font\Bbbsmallfont=msbm8%
\textfont\Bbbfam=\Bbbfont\scriptfont\Bbbfam=\Bbbsmallfont

\begin{document}
\setlength{\parskip}{2pt plus 4pt minus 0pt}
\hfill {\scriptsize May 5, 2013} 
\vskip1.5ex

\title[Abelian ideals and root systems]{Abelian ideals of a Borel subalgebra 
and root systems}
\author{Dmitri I. Panyushev}
\address[]{Independent University of Moscow,
Bol'shoi Vlasevskii per. 11, 119002 Moscow,  Russia
\hfil\break\indent
Institute for Information Transmission Problems of the R.A.S., B. Karetnyi per. 19,  
127994 Moscow,  Russia}
\email{panyushev@iitp.ru}
\begin{abstract}

Let $\g$ be a simple Lie algebra  
and  $\Abo$  the poset of non-trivial abelian ideals of a fixed Borel subalgebra of $\g$. 
In~\cite{imrn}, we constructed a partition $\Abo=\sqcup_\mu \Ab_\mu$ parameterised by 
the long positive roots of $\g$ and studied the subposets $\Ab_\mu$. In this note, we show that 
this partition is compatible with intersections, relate it to the Kostant-Peterson parameterisation and 
to the centralisers of abelian ideals. We also prove that the poset of positive roots of $\g$ is a join-semilattice.
\end{abstract}
\maketitle

\section*{Introduction}

\noindent
Let $\g$ be a complex simple Lie algebra with a  triangular decomposition 
$\g=\ut\oplus\te\oplus \ut^-$. Here $\te$ is a fixed Cartan subalgebra and $\be=\ut\oplus\te$
is a fixed Borel subalgebra. Accordingly, $\Delta$ is the set of roots of $(\g,\te)$, $\Delta^+$
is the set of positive roots corresponding to $\ut$, and $\Pi$ is the set of simple roots in  
$\Delta^+$. Write $\theta$ for the highest root in  $\Delta^+$.

A subspace $\ah\subset\ut$ is  an {\it abelian ideal\/} (of $\be$) if 
$[\be,\ah]\subset \ah$ and $[\ah,\ah]=0$. 
The set of abelian ideals of $\be$ is denoted by $\Ab$.
In the landmark paper~\cite{ko98},  Kostant  elaborated on 
Dale Peterson's theory of Abelian ideals (in particular, the astounding result 
that $\#\Ab=2^{\rk\g}$) and related abelian ideals with problems in representation theory. Since then, 
abelian ideals attracted a lot of attention,~see e.g.
\cite{cp1,cp2,cp3,pr,imrn,subsets,suter}. 
We think of $\Ab$ as a poset with respect to inclusion.
As $\ah\in\Ab$ is a sum of certain root spaces, we may (and will) identify such $\ah$ 
with the corresponding subset $I=I_\ah$ of $\Delta^+$.

Let $\Abo=\Abo(\g)$ denote the set of nonzero abelian ideals and $\Delta^+_l$  the set
of long positive roots. In the simply-laced case, all roots are assumed to be long.
In \cite[Sect.\,2]{imrn}, we defined a surjective mapping 
$\tau: \Abo \to \Delta^+_l$ and studied its fibres. 
If $\ah\in \Abo$ and $\tau(\ah)=\mu$, then $\mu$ is called the 
{\it rootlet\/} of $\ah$, also denoted by $\rt(\ah)$ or $\rt(I_\ah)$.  
Letting $\Ab_\mu=\tau^{-1}(\mu)$, we get a  partition of $\Abo$ parameterised by
$\Delta^+_l$. Each fibre $\Ab_\mu$ is regarded as a 
sub-poset of $\Ab$.  It is known that, for any $\mu\in \Delta^+_l$,
$\Ab_\mu$ has a unique minimal and unique maximal  element \cite[Sect.\,3]{imrn}. 
Regarding abelian ideals as subsets of $\Delta^+$, we write $I(\mu)_{\min}$ (resp. $I(\mu)_
{\max}$) for the minimal (resp. maximal) element of  $\Ab_\mu$.
We also say that 
$I(\mu)_{\min}$ is the $\mu$-{\it minimal\/} and $I(\mu)_{\max}$ is the 
$\mu$-{\it maximal\/} ideal.
Various properties of the $\mu$-minimal ideals are obtained in \cite[Sect.\,4]{imrn}. For instance,
if $(\ ,\ )$ is a Weyl group invariant scalar product, then   
\begin{itemize}
\item \ $\# I(\mu)_{\min}=(\rho, \theta^\vee-\mu^\vee)+1$, where $\rho=\frac{1}{2}
      \sum_{\gamma\in\Delta^+} \gamma$ and $\mu^\vee=2\mu/(\mu,\mu)$;
\item \   $I=I(\mu)_{\min}$  for some $\mu\in\Delta^+_l$ if and only if $I\subset
      \eus H:=\{\gamma\in \Delta^+ \mid (\gamma, \theta)\ne 0\}$;
\item \ $I(\mu)_{\min} \subset I(\mu')_{\min}$ if and only if $\mu'\curle \mu$, where `$\curle$' is the usual {\it root order\/} on $\Delta^+$.
\item \ $I(\mu)_{\min}=I(\mu)_{\max}$ if and only if $(\mu,\theta)=0$ \cite[Thm.\,5.1]{imrn}.
\end{itemize}
If $\rt(I)\not\in \Pi$,  then there is $I'\in \Ab$ such that 
$I'\supset I$, $\#I'=\#I+1$ and $\rt(I')\prec \rt(I)$. This is implicit in~\cite[Thm.\,2.6]{imrn},
cf. also Proposition~\ref{prop:elem-ext}. This implies that the (globally) maximal ideals of $\Ab$ are precisely the maximal elements of the posets $\Ab_{\ap}$ for $\ap\in \Pi\cap\Delta^+_l=:\Pi_l$, see \cite[Cor.\,3.8]{imrn}. A closed formula for the dimension of all maximal abelian ideals is proved in
\cite[Sect.\,8]{cp3},\cite{suter}.
In this paper, we elaborate on further properties of the partition
\beq  \label{eq:parti}
  \Abo=\sqcup_{\mu\in\Delta^+_l}\Ab_\mu \ 
\eeq
and related properties of abelian ideals and root systems.

In Section~\ref{sect:intersect}, we show that  partition \eqref{eq:parti} behaves well with respect 
to intersections.

\begin{thm}  \label{thm:intr2} 
Let $\mu,\mu'\in\Delta^+_l$.
\begin{itemize}
\item[\sf (i)] \ If $I \in \Ab_\mu$ and $I'\in \Ab_{\mu'}$, then $I\cap I'$ belongs to
$\Ab_{\nu}$, where $\nu$ does not depend on the choice of $I$ and $I'$. Actually,
$\nu$ is the unique smallest long positive root such that $\nu\curge \mu$
and $\nu\curge \mu'$. In particular, such $\nu$ always exists;
\item[\sf (ii)] \ Furthermore,   $I(\mu)_{\min}\cap I(\mu')_{\min}=I(\nu)_{\min}$,
$I(\mu)_{\max}\cap I(\mu')_{\max}=I(\nu)_{\max}$, and every ideal in $\Ab_\nu$ occurs as 
intersection of two ideals from $\Ab_\mu$ and $\Ab_{\mu'}$.
\end{itemize}
\end{thm}
\noindent
The root $\nu$ occurring in (i) is denoted by $\mu\vee\mu'$.
In our approach, the existence of $\mu\vee\mu'$ \ ($\mu,\mu'\in \Delta^+_l$) comes up as a 
by-product of our theory of posets $\Ab_\mu$.
This prompts the natural question of whether `$\vee$' is well-defined for 
{\sl all\/} pairs of positive roots, not necessarily long.
The corresponding general  assertion is proved in the Appendix (see Theorem~\ref{thm-app2}). 
It seems that this property of root systems has not been noticed before.

In Section~\ref{sect:sovpad}, we give a characterisation of $\mu$-minimal abelian ideals that relates two different approaches to $\Ab$. We have associated the rootlet $\rt(I)\in
\Delta^+_l$ to a nonzero abelian ideal $I$. On the other hand, there is a bijection between
$\Ab$ and certain elements in the coroot lattice $Q^\vee$, which is due to Kostant and Peterson~\cite{ko98}.  Namely, 
\[
   \Ab \stackrel{1:1}{\longleftrightarrow} \eus Z_1=\{z\in Q^\vee \mid -1\le (z,\gamma)\le 2 \text{ for all }
\gamma\in \Delta^+\} .
\]
The element $z\in Q^\vee$ corresponding to $I\in\Ab$ is denoted by $z_I$. Our result is 
\begin{thm}  \label{thm:intr1}
For an abelian ideal $I$, we have

 $I=I(\mu)_{\min}$ for $\mu=\rt(I)$ if and only if\/ $\rt(I)^\vee=z_I$.
\end{thm}

\noindent We also prove that 
\\ \indent
\textbullet \quad an abelian ideal $I$ belongs to $\Ab_\mu$ if and only if 
$I\cap\gH=I(\mu)_{\min}$; 
\\ \indent
\textbullet \quad  $I(\mu)_{\max}\subset \{\nu\in\Delta^+\mid \nu\curge\mu\}$.
 
In Section~\ref{sect:central},  we consider the centralisers of abelian ideals. If $\ah\in\Ab$, 
then the centraliser $\z_\g(\ah)$ is a $\be$-stable subspace of $\g$. 
However, $\z_\g(\ah)$ is not always contained in $\be$. We give criteria for $\z_\g(\ah)$ to
be a nilpotent subalgebra or a sum of abelian ideals. We also prove

\begin{thm}  \label{thm:intr3}  Let $\ah\in\Ab$.
Then $\z_\g(\ah)$ is again an abelian ideal if and only if $\rt(\ah)\in \Pi_l$.
In particular, $\z_\g(\ah)=\ah$ if and only if  $\ah$ is a  maximal  ideal in $\Ab$.
\end{thm}

\noindent
In fact, Theorem~\ref{thm:intr3} is closely related to the following interesting observation.
For any $\gS\subset \Delta^+$, let $\min(\gS)$ and $\max(\gS)$ denote the sets of minimal
and maximal elements of $\gS$, respectively.

\begin{thm}  \label{thm:intr4}
For every $\ap\in\Pi_l$, there is a one-to-one correspondence between
$\min\bigl( I(\ap)_{\min}\bigr)$ and $\max \bigl(\Delta^+\setminus I(\ap)_{\max}\bigr)$.
Namely, if
$\nu \in \min\bigl( I(\ap)_{\min}\bigr)$, then $\theta-\nu\in 
\max \bigl(\Delta^+\setminus I(\ap)_{\max}\bigr)$
; and vice versa. In particular, $\max \bigl(\Delta^+\setminus I(\ap)_{\max}\bigr)\subset \gH$. 
\end{thm}

An analogous statement for arbitrary long roots (in place of $\ap\in\Pi_l$) is not true. 
However, there is a modification of Theorem~\ref{thm:intr4} 
that applies to the connected subsets of $\Pi_l$, 
see Theorem~\ref{thm:modification}. 

\noindent
We refer to \cite{bour,hump} for standard results on root systems and (affine) Weyl groups.

{\small
{\bf Acknowledgements.} This work was done during my visits to CRC~701 at
the Universit\"at Bielefeld and 
Max-Planck-Institut f\"ur Mathematik (Bonn). I thank both Institutions for
the hospitality  and support. 
I am grateful to E.B.\,Vinberg for fruitful discussions related to results in Appendix~\ref{app:A}.
Thanks are also due to the anonymous referee for providing the outline of a uniform proof of 
Theorem~\ref{thm:intr4}.
}

\section{Preliminaries on abelian ideals and minuscule elements}
\label{sect:odin}

\noindent
Throughout this paper, $\Delta$ is the root system of $(\g,\te)$ with positive roots 
$\Delta^+$ corresponding to $\ut$, simple roots $\Pi=\{\ap_1,\dots,\ap_n\}$, and
Weyl group $W$. Set $\Pi_l:=\Pi\cap \Delta^+_l$.
We equip  $\Delta^+$  with the usual partial ordering `$\curle$'.
This means that $\mu\curle\nu$ if $\nu-\mu$ is a non-negative integral linear combination
of simple roots. Write $\mu\prec\nu$ if $\mu\curle\nu$ and $\mu\ne\nu$.

If $\ah$ is an abelian ideal of $\be$, then $\ah$ is a sum of certain root spaces in $\ut$,  i.e.,
$\ah=\bigoplus_{\gamma\in I_\ah}\g_\gamma$.  The relation $[\be,\ah]\subset \ah$ is equivalent to
that $I=I_\ah$ is an {\it upper ideal\/} of the poset $(\Delta^+, \curle)$, i.e., 
if $\nu\in I$, $\gamma\in\Delta^+$, and $\nu\curle \gamma$, then $\gamma\in I$.
The property of being abelian means that
$\gamma'+\gamma''\not\in \Delta^+$ for all $\gamma',\gamma''\in I$.
We often work in the setting of root systems, so that a $\be$-ideal $\ah\subset\ut$
is being identified with the corresponding subset $I$ of positive roots. 

The theory of abelian ideals  relies on the relationship, due to Peterson, between the 
abelian ideals and the so-called {\it minuscule elements\/} of the affine Weyl 
group of $\Delta$. Recall the necessary setup.

We have the vector space $V=\oplus_{i=1}^n{\mathbb R}\ap_i$, 
the  Weyl group $W$ generated by  simple reflections
$s_1,\dots,s_n$,  and a $W$-invariant inner product $(\ ,\ )$ on $V$. 
Letting $\widehat V=V\oplus {\mathbb R}\delta\oplus {\mathbb R}\lb$, we extend
the inner product $(\ ,\ )$ on $\widehat V$ so that $(\delta,V)=(\lb,V)=
(\delta,\delta)= (\lb,\lb)=0$ and $(\delta,\lb)=1$. Set  $\ap_0=\delta-\theta$, where
$\theta$ is the highest root in $\Delta^+$. 
Then

\begin{itemize}
\item[] \ 
$\widehat\Delta=\{\Delta+k\delta \mid k\in {\mathbb Z}\}$ is the set of affine
(real) roots; 
\item[] \ $\HD^+= \Delta^+ \cup \{ \Delta +k\delta \mid k\ge 1\}$ is
the set of positive affine roots; 
\item[] \ $\HP=\Pi\cup\{\ap_0\}$ is the corresponding set
of affine simple roots;
\item[] \  $\mu^\vee=2\mu/(\mu,\mu)$ is the coroot corresponding to 
$\mu\in \widehat\Delta$;
\item[] \   $Q=\oplus _{i=1}^n {\mathbb Z}\ap_i$  
is the {\it root lattice\/}  and $Q^\vee=\oplus _{i=1}^n {\mathbb Z}\ap_i^\vee$  
is the {\it coroot lattice\/} in $V$.
\end{itemize}
 
\noindent
For each $\ap_i\in \HP$, let $s_i$ denote the corresponding reflection in $GL(\HV)$.
That is, $s_i(x)=x- (x,\ap_i)\ap_i^\vee$ for any $x\in \HV$.
The affine Weyl group, $\HW$, is the subgroup of $GL(\HV)$
generated by the reflections $s_0,s_1,\dots,s_n$.
The extended inner product $(\ ,\ )$ on $\widehat V$ is $\widehat W$-invariant. 
The {\it inversion set\/} of $w\in\HW$ is $\eus N(w)=\{\nu\in\HD^+\mid w(\nu)\in -\HD^+\}$.

Following Peterson, we say that $w\in \HW$ is  {\it minuscule\/}, if 
$\eus N(w)=\{-\gamma+\delta\mid \gamma\in I_w\}$ 
for some subset $I_w\subset \Delta$.
One then proves that {\sf (i)} $I_w\subset \Delta^+$, {\sf (ii)} $I_w$ is an abelian ideal, and
{\sf (iii)} the assignment 
$w\mapsto I_w$ yields a bijection between the minuscule elements of
$\HW$ and the abelian ideals, see \cite{ko98},  \cite[Prop.\,2.8]{cp1}. 
Accordingly, if $I\in\Ab$, then $w_I$ denotes the corresponding minuscule
element of $\HW$. Obviously, $\# I=\#\eus N(w_I)=\ell(w_I)$, where $\ell$ is the usual length function on $\HW$.

Using minuscule elements of $\HW$, one can assign an element of $Q^\vee$ 
to any abelian ideal~\cite{ko98}. 
In fact, one can associate an element of $Q^\vee$ to any $w\in\HW$.
The following is exposed in a more comprehensive form  in \cite[Sect.\,2]{losh}.

\noindent
Recall that $\HW$ is a semi-direct product of $W$ and $Q^\vee$, and it can also be regarded as
a group of affine-linear transformations of $V$ \cite[4.2]{hump}.
For any $w\in \HW$, there is a unique decomposition
\beq       \label{eq:decomp-w}
   w=v{\cdot}t_r,
\eeq
where $v\in W$ and $t_r$ is the translation of $V$ corresponding to $r\in Q^\vee$, i.e., 
$t_r\ast x=x+r$ for all $x\in V$.
Then we assign the element $v(r)\in Q^\vee$ to  $w\in\HW$.
An alternative way for doing so, which does not explicitly use the semi-direct product 
structure, is based on the relation between 
the linear $\HW$-action on $\HV$ and  decomposition \eqref{eq:decomp-w}.
Given $w\in\HW$, 
define the integers $k_i$, $i=1,\dots,n$,  by the formula
$w^{-1}(\ap_i)=\mu_i+k_i\delta$ ($\mu_i\in \Delta$).
Then $v(r)\in Q^\vee$ is determined by the conditions  that $(v(r),\ap_i)=k_i$.
The reason is that $w^{-1}=v^{-1}{\cdot}t_{-v(r)}$ and 
the linear $\HW$-action on $\HV$ satisfies the following relation
\beq   \label{eq:general-aff-lin}
    w^{-1}(x)=v^{-1}(x)+(x,v(r))\delta  \quad \forall x\in V\oplus\BR\delta .
\eeq
It suffices to verify that $t_r(x)=x-(x,r)\delta$.

If $w=w_I$ is minuscule, then we also write $z_I$ for the resulting element of $Q^\vee$. 
By \cite[Theorem\,2.5]{ko98}, the mapping
$I \mapsto z_I\in V$ sets up a bijection between $\Ab$
and   $\eus Z_1=\{ z\in Q^\vee \mid (z,\gamma)\in \{-1,0,1,2\}  \quad \forall \gamma\in \Delta^+\}$.  
A proof of this result is given in \cite[Appendix~A]{subsets}.

Given $I\in\Abo$ and the corresponding non-trivial minuscule element $w_I\in\HW$, 
the {\it rootlet\/} of $I$ is defined by 
\[
   \rt(I)=w_I(\ap_0)+\delta=w_I(2\delta-\theta) .
\]
By \cite[Prop.\,2.5]{imrn}, we have $\rt(I)\in \Delta^+_l$.
The next result  describes a procedure for extensions of abelian ideals. Namely, if 
the rootlet of $I=I_w$ is not simple, then one can construct a larger ideal $I'$ such that 
$\# I'=\# I+1$ and $\rt(I')=s_\ap(\rt(I))\prec \rt(I)$ for some $\ap\in\Pi$.

\begin{prop}           \label{prop:elem-ext}
Let $w\in\HW$ be minuscule and $\mu=\rt(I_w)$.
Suppose that $\mu\not\in\Pi$ and 
take any $\ap\in\Pi$ such that $(\ap,\mu)>0$. Then $s_\ap w$ is again minuscule.
Moreover, the only root in $I_{s_\ap w}\setminus I_w$  belongs to $\gH$.
\end{prop}
\begin{proof}
Set $\mu'=s_\ap(\mu)=s_\ap w(2\delta-\theta)$ and $\mu''=\mu-\ap$. (Note that
$\mu'=\mu''$ if and only if $\ap\in\Pi_l$).  Then $w(2\delta-\theta)=\mu''+\ap$ and
$w^{-1}(\mu'')+w^{-1}(\ap)=2\delta-\theta$. Therefore,

$\begin{cases} w^{-1}(\mu'')=k\delta-\mu_1 & \\
  w^{-1}(\ap)=(2-k)\delta-\mu_2 & \end{cases}$, where $\mu_1,\mu_2\in \Delta$ and 
$\mu_1+\mu_2=\theta$. 
\\[.6ex]
This clearly implies that both $\mu_1$ and $\mu_2$ are positive and
hence $\mu_1,\mu_2\in \gH$.
Furthermore, since $w$ is minuscule, both $w^{-1}(\mu'')$ and $w^{-1}(\ap)$ must be 
positive. [Indeed, if, say, $w^{-1}(\mu'')$ is negative, then $k\le 0$. Hence $w(\mu_1)=
k\delta-\mu''$ is negative and $\mu_1\in\eus N(w)$, which contradicts the definition of 
minuscule elements.]
Therefore, one must have $k=1$. Then 
$w(\delta-\mu_2)=\ap\in\Pi$. Since $\eus N(s_\ap w)=\eus N(w)\cup \{w^{-1}(\ap)\}$,
we then conclude that 
$s_\ap w$ is minuscule and the corresponding abelian ideal is
$I_{s_\ap w}=I_w\cup\{\mu_2\}$. 
\\
Note also that $\rt(I_{s_\ap w})=\mu'\prec \mu$.
\end{proof}

\section{Intersections of abelian ideals and posets $\Ab_\mu$}
\label{sect:intersect} 

\noindent
In this section, we prove that taking intersection of abelian ideals is compatible with partition~\eqref{eq:parti}.

First of all, we notice that for any collection of non-empty abelian ideals (subsets of $\Delta^+$)
their intersection
is non-empty, since all these ideals contain the highest root $\theta$. In particular, 
if $\mu_1,\dots,\mu_s\in \Delta^+_l$, then
\[
   I=\bigcap_{i=1}^s I(\mu_i)_{\min}
\]
is again an abelian ideal. Since $I(\mu_i)_{\min}\subset \gH$ for all $i$, we have 
$I\subset \gH$, and therefore $I=I(\mu)_{\min}$ for certain $\mu\in \Delta^+_l$ \cite[Thm.\,4.3]{imrn}.
Since $I(\mu)_{\min}\subset I(\mu_i)_{\min}$, we conclude that $\mu\curge\mu_i$
\cite[Cor.\,3.3]{imrn}.

On the other hand, if $\gamma\in\Delta^+_l$ and $\gamma\curge\mu_i$ for all $i$, then 
$I(\gamma)_{\min}\subset I(\mu_i)_{\min}$ \cite[Thm.\,4.5]{imrn}. Therefore, $I(\gamma)_{\min}\subset I(\mu)_{\min}$, i.e., 
$\gamma\curge\mu$.  Thus, we have proved

\begin{thm}  \label{thm:sup-min}
For any collection $\mu_1,\dots,\mu_s\in \Delta^+_l$, 

{\sf (i)} \ there exists a unique long root $\mu$ such that $\mu\curge\mu_i$ for all $i$, and if $\gamma\in\Delta^+_l$ and $\gamma\curge\mu_i$
for all $i$, then $\gamma\curge\mu$;

{\sf (ii}) \ $\bigcap_{i=1}^s I(\mu_i)_{\min}=I(\mu)_{\min}$.
\end{thm}

The root $\mu$ occurring in part (i) is denoted by $\mu_1\vee\ldots\vee\mu_s
=\vee_{i=1}^s\mu_i$. We also say that $\mu$ is the {\it least upper bound\/} or
{\it join\/} of $\mu_1,\dots,\mu_s$.

\begin{rmk} 
Clearly, the operation `$\vee$' is associative, and 
it suffices to  describe the least upper bound for only  two (long) roots.
In Appendix~\ref{app:A}, we prove directly that the join exists for {\sl all\/} pairs of roots, not 
necessarily long ones, and
give an explicit formula for it.
\end{rmk}

We are going to play the same game with arbitrary ideals in $\Ab_{\mu_i}$. 
To this end, we need an analogue of
\cite[Thm.\,4.5]{imrn} for the $\mu$-maximal ideals, see Corollary~\ref{cor:1}(i) below. 
This can be achieved as follows.

\begin{prop}                  \label{prop:long-ext}
Let $\mu,\mu'$ be long roots such that $\mu'\prec \mu$. Then
\begin{itemize}
\item[{\sf (i)}] \ for any $I\in\Ab_\mu$, there exists $I'\subset \Ab_{\mu'}$ such that 
$I'\supset I$ and $\# I'=\# I+(\rho,\mu^\vee-{\mu'}^\vee)$;
\item[{\sf (ii)}] \ moreover, if $I=I_0\subset I_1\subset \ldots\subset I_m=I'$ is any chain of 
ideals with $m=(\rho,\mu^\vee-{\mu'}^\vee)$ and 
$\# I_j=\# I_{j-1}+1$, then $\rt(I_j)\ne \rt(I_{j-1})$ for all $j$.
\end{itemize}
\end{prop}
\begin{proof}
If $\mu\not\in\Pi_l$ and $\ap\in\Pi$ with $(\ap,\mu)>0$, then a direct calculation shows
that  $(\rho,\mu^\vee-s_\ap(\mu)^\vee)=1$. [Use the relations $(\rho,\ap^\vee)=1$ and
$(\ap,\mu^\vee)=1$.]

(i) \ Arguing by induction, one readily proves that if $\mu,\mu'$ are both long and 
$\mu'\prec\mu$, then $\mu'$ can be reached
from $\mu$ by a sequence of simple reflections:
\[
  \mu=\mu_0\to s_{\gamma_1}(\mu_0)=\mu_1\to s_{\gamma_2}(\mu_1)=\mu_2\to \ldots \to 
  s_{\gamma_m}(\mu_{m-1})=\mu_m=\mu' ,
\]
where $\gamma_i\in\Pi$ and $(\gamma_i,\mu_{i-1})>0$. The number of steps $m$ equals
$(\rho,\mu^\vee-{\mu'}^\vee)$. If $I\in\Ab_{\mu}$ is arbitrary and $w_I$ is the corresponding 
minuscule element, then the repeated application of Proposition~\ref{prop:elem-ext} shows that
$w':=s_{\gamma_1}\ldots s_{\gamma_m}w_I$ is again minuscule and $I'=I_{w'}$ is a required ideal.

(ii) \  Let $w_j\in \HW$ be the minuscule element corresponding to $I_j$. 
Then $w_j=s_{i_j}w_{j-1}$ for a sequence $(\ap_{i_1},\dots,\ap_{i_m})$ of affine simple roots.
The corresponding sequence of rootlets is
\[
  \mu=\mu_0 \to s_{i_1}\mu_0=\mu_1\to   s_{i_2}\mu_1=\mu_2 \to \dots \to \mu_m=\mu' . 
\]
If $i_j=0$, i.e., the $j$-th step is the reflection with the respect to $\ap_0=\delta-\theta$, then 
$\mu_{j-1}=\mu_j$,  see \cite[Prop.\,3.2]{imrn}.
For the steps corresponding to $\ap_{i_j}\in\Pi$, the value of $(\rho,\mu_j^\vee)$ is reduced
by at most $1$. Consequently, the sequence $(\ap_{i_1},\dots,\ap_{i_m})$ does not contain
$\ap_0$ and the value of $(\rho,\mu_j^\vee)$ decrease by $1$ at each step, i.e., all these
rootlets are different.
\end{proof}

\begin{cl}                  \label{cor:1}
If $\mu,\mu'$ are long roots such that $\mu'\curle \mu$, then
\begin{itemize}
\item[{\sf (i)}] \  $I(\mu)_{\max}\subset I(\mu')_{\max}$;
\item[{\sf (ii)}]  \ $\#\Ab_{\mu'} \ge \#\Ab_\mu$.
\end{itemize}
\end{cl}
\begin{proof}
(i) \ This readily follows from Proposition~\ref{prop:long-ext}(i) applied to $I=I(\mu)_{\max}$.
\\
(ii) \ Argue by induction on $m=(\rho,\mu^\vee-{\mu'}^\vee)$. For $m=1$, the assertion 
follows from Proposition~\ref{prop:elem-ext}.
\end{proof}

\begin{thm}           \label{thm:sup-max}
For any set $\{\mu_1,\dots,\mu_s\}\subset \Delta^+_l$ and $\mu=\vee_{i=1}^s\mu_i$,  we have 
\begin{itemize}
\item[{\sf (i)}] \  $\bigcap_{i=1}^s I(\mu_i)_{\max}=I(\mu)_{\max}$,
\item[{\sf (ii)}] \ If $I_i\in \Ab_{\mu_i}$ for $i=1,\dots,s$, then $\bigcap_{i=1}^s I_i\in \Ab_\mu$. 
\item[{\sf (iii)}] \ For every $I\in \Ab_\mu$, there exist $I_i\in\Ab_{\mu_i}$ such that 
$I=\bigcap_{i=1}^s I_i$. 
\end{itemize}
\end{thm}
\begin{proof}
(i) \ Consider the abelian ideal $I= \bigcap_{i=1}^s I(\mu_i)_{\max}$. Since $I\subset I(\mu_i)_{\max}$, we have $\rt(I)\curge \mu_i$ for all $i$, hence $\rt(I)\curge \vee_{i=1}^s\mu_i=\mu$.
We also have $I\supset \bigcap_{i=1}^s I(\mu_i)_{\min}=I(\mu)_{\min}$, hence
$\rt(I)\curle \mu$ by \cite[Cor.\,3.3]{imrn}. 
It follows that $\rt(I)=\mu$ and $I\subset I(\mu)_{\max}$.

Since $\mu\curge\mu_i$,  by Corollary~\ref{cor:1}(i), we have
$I(\mu)_{\max}\subset I(\mu_i)_{\max}$ for all $i$, and $I(\mu)_{\max}\subset I$.

Thus, $I=I(\mu)_{\max}$.

(ii) \ It follows from Theorem~\ref{thm:sup-min}(ii) and  part (i) that
$I(\mu)_{\min} \subset \bigcap_{i=1}^s I_i \subset I(\mu)_{\max}$.
By \cite[Thm.\,3.1(iii)]{imrn}, the intermediate ideal $\bigcap_{i=1}^s I_i $ also belongs to $\Ab_\mu$.

(iii) \ 
Given $I\in \Ab_\mu$, we  construct the ideals $I_i\in \Ab_{\mu_i}$,
$i=1,\dots,s$, as prescribed in Proposition~\ref{prop:long-ext}(i).
Then $I\subset \bigcap_{i=1}^s I_i=:J$ and $\rt(J)=\vee_{i=1}^s \mu_i=\mu$.
That is, $\rt (I)=\rt(J)$. By Proposition~\ref{prop:long-ext}(ii), this is only possible if
$J=I$.
\end{proof}

Combining Theorems~\ref{thm:sup-min} and \ref{thm:sup-max} yields 
Theorem~\ref{thm:intr2} in the Introduction.
 
For any $\gamma\in\Delta^+$, set $I\langle{\curge}\gamma\rangle=\{\nu\in\Delta^+\mid
\nu\curge\gamma\}$. We also say that  $I\langle{\curge}\gamma\rangle$ 
is the {\it principal\/} upper ideal of $\Delta^+$ {\it generated\/} by $\gamma$.
It is not necessarily abelian.

\begin{ex}  \label{ex:all-maximal}
Let $\ap_1,\dots,\ap_s$ be the set of all long simple roots. 
Then $\vee_{i=1}^s\ap_i=\sum_{i=1}^s\ap_i=\vert\Pi_l\vert$ and 
$\{ I(\ap_i)_{\max}\mid i=1,\dots,s\}$ is the set of all  maximal abelian ideals in $\Ab$.
Hence $\bigcap_{i=1}^s I(\ap_i)_{\max}$ is an ideal with rootlet
$\vert\Pi_l\vert$.
Inspecting the list of root systems, we notice that the ideal
$\bigcap_{i=1}^s I(\ap_i)_{\min}=I(\vert\Pi_l\vert)_{\min}$ has a nice uniform description.
For any $\gamma=\sum_{i=1}^n a_i\ap_i\in\Delta^+$, we set $[\gamma/2]=\sum_{i=1}^n [a_i/2]\ap_i$. Then 
$I(\vert\Pi_l\vert)_{\min}$ is the upper ideal of $\Delta^+$ generated by the root $\theta-[\theta/2]$. (It is true that $\theta-[\theta/2]$ is always a root in $\gH$.)

In the {\bf A-D-E} case, we have $\vert \Pi_l\vert=|\Pi |$ and hence $(\theta, |\Pi_l |)\ne 0$. 
In fact, $(\theta, |\Pi_l |)\ne 0$ for all simple Lie algebras except type $\GR{C}{n}$, $n\ge 2$.
The condition $(\theta, |\Pi_l |)\ne 0$ implies that $\# \Ab_{|\Pi_l |}=1$ \cite[Thm.\,5.1]{imrn}, i.e., 
$I(|\Pi_l |)_{\min}=I(|\Pi_l |)_{\max}$  if $\g$ is not of type $\GR{C}{n}$.
\end{ex}

\begin{rmk}          \label{rmk:comm-roots}
The interest in $[\theta/2]$ is also justified by the following observations. 
As in \cite{rodstv}, we say that $\gamma\in\Delta^+$ is {\it commutative}, if the 
$\be$-submodule of $\g$ generated by $\g_\gamma$ is an abelian ideal; equivalently, if 
the upper ideal $I\langle{\curge}\gamma\rangle$
is abelian. Let $\Delta^+_{\mathsf{com}}$ denote the set of all commutative roots.
Clearly, $\Delta^+_{\mathsf{com}}=\bigcup_{\ap_i\in\Pi_l}I(\ap_i)_{\max}$.
It was noticed in \cite[Thm.\,4.4]{rodstv} that $\Delta^+\setminus \Delta^+_{\mathsf{com}}$ has a 
unique maximal element, and this maximal element is $[\theta/2]$.

For any $\gamma\in\Delta^+$, it appears to be  true that $[\gamma/2]\in \Delta^+\cup\{0\}$ and 
$\gamma-[\gamma/2]\in\Delta^+$. It would be interesting  to have a conceptual explanation 
for this. 
\end{rmk}

\section{Some properties of posets $\Ab_\mu$} 
\label{sect:sovpad}

Let $I\subset \Delta^+$ be an abelian ideal and $w_I=v{\cdot}t_r\in\HW$ the 
corresponding minuscule element. Recall that $v\in W$ and $r\in Q^\vee$.
We have associated two objects to these data:
the rootlet $\rt(I)=w_I(2\delta-\theta)\in \Delta^+_l\subset Q$ and the  element
$z_I:=v(r)\in Q^\vee$.

\begin{thm}
For an abelian ideal $I$, the following conditions are equivalent:
\begin{itemize}
\item[\sf (i)]  \ $\rt(I)^\vee=z_I$;
\item[\sf (ii)] \ $I=I(\mu)_{\min}$ \ for $\mu=\rt(I)$.
\end{itemize}
\end{thm}
\begin{proof}
1) \ Suppose that $I=I(\mu)_{\min}$. By \cite[Thm.\,4.3]{imrn}, 
$w_I=v_\mu s_0$, where $v_\mu \in W$ is the unique element of minimal length such
that  $v_\mu(\theta)=\mu$. Here $\ell(v_\mu)=(\rho,\theta^\vee-\nu^\vee)$. It is easily seen that for $w=s_0$ 
decomposition~\eqref{eq:decomp-w} is  $s_0=s_\theta{\cdot}t_{-\theta^\vee}$, where 
$s_\theta\in W$ is the reflection with respect to $\theta$. Hence 
the linear part of $w_I$ is $v_\mu s_\theta$ and
$r=-\theta^\vee$.
Therefore, $v_\mu s_\theta(-\theta^\vee)=v_\mu(\theta^\vee)=\mu^\vee$, as required.

2) \ Conversely, if $\rt(I)=\mu$ and $I\ne I(\mu)_{\min}$, then $z_I\ne z_{I(\mu)_{\min}}$.
By the first part, we have 
$z_{I(\mu)_{\min}}=\mu^\vee $. Thus, $z_I\ne \rt(I)^\vee$.
\end{proof}

Applying formulae~\eqref{eq:decomp-w} and \eqref{eq:general-aff-lin} to arbitrary minuscule $w_I$, 
we obtain
\[
    w_I(2\delta-\theta)=v{\cdot}t_r(2\delta-\theta)= v(2\delta+(\theta,r)\delta-\theta)=
    -v(\theta)+(2+(\theta,r))\delta .
\]
As we know that $\rt(I)\in \Delta^+$, one must have $(\theta,r)=-2$ and 
$-v(\theta)\in\Delta^+$. Therefore, the equality $\rt(I)^\vee=z_I$ is equivalent to that
$v(r)=-v(\theta^\vee)$, i.e., $r=-\theta^\vee$.

This can be summarised as follows:

\noindent
{\it If $w_I=v{\cdot}t_r\in \HW$ is minuscule, then $(\theta,r)=-2$. Moreover,
$\rt(I)^\vee=z_I$ if and only if $r=-\theta^\vee$, i.e., $r$ is the shortest element in the affine hyperplane $\{x\in V\mid (x,\theta)=-2\}$.}

The theory developed in \cite[Sect.\,4]{imrn} yields, in principle, a very good understanding of 
$\mu$-minimal ideals. In particular, an abelian ideal $I$ is minimal in some $\Ab_\mu$ if and only if
$I\subset \gH=\{\nu\in\Delta^+ \mid (\nu,\theta)\ne 0\}$ \cite[Thm.\,4.3]{imrn}.
The other ideals in $\Ab_\mu$ can be characterised as follows.

\begin{prop}    \label{prop:other-in-fibre}
For $\mu\in\Delta^+_l$ and $I\in\Ab$, we have
 $I\in\Ab_\mu$ if and only if $I\cap\gH=I(\mu)_{\min}$.
\end{prop}
\begin{proof}
($\Rightarrow$) Since $I(\mu)_{\min}\subset \gH$, we have $I(\mu)_{\min}\subset I\cap\gH$.
Moreover, $I\cap\gH=I(\mu')_{\min}$ for some $\mu'\in\Delta^+_l$.
Then  $I(\mu)_{\min}\subset I(\mu')_{\min}\subset I$. By \cite[Cor.\,3.3]{imrn}, this yields
opposite inequalities for the rootlets, i.e.,
$\mu\curge\mu'\curge \mu$.

($\Leftarrow$) If $\mu'=\rt(I)$, then $I\cap\gH=I(\mu')_{\min}$ according to the previous part.
Hence $\mu'=\mu$.
\end{proof}

This implies that all ideals in $\Ab_\mu$ can be obtained from $I(\mu)_{\min}$ by adding 
suitable roots outside $\gH$. In particular, 
$I(\mu)_{\max}$ is maximal among all abelian ideals having the prescribed intersection,
$I(\mu)_{\min}$, with $\gH$.

For future use, we provide a property of ideals $I(\ap)_{\min}$ with $\ap\in\Pi_l$. Recall that the
integer $1+(\rho,\theta^\vee)=h^*$ is called the {\it dual Coxeter number\/} of $\Delta$.
By \cite[Prop.\,1]{rudi}, $\#\gH=2h^*-3$. If $\gamma\in\gH\setminus\{\theta\}$, then $\theta-\gamma
\in\gH\setminus\{\theta\}$ as well. Hence $\gH$ can be presented as the disjoint union of
$\theta$ and $h^*-2$ pairs of the form $\{\gamma_i,\theta-\gamma_i\}$.
By \cite[Theorem\,3.1]{imrn}, we have $\# I(\ap)_{\min}=(\rho,\theta^\vee-\ap^\vee)+1=
(\rho,\theta^\vee)=h^*-1$. Because $I(\ap)_{\min} \subset \gH$, we see that $I(\ap)_{\min}$ must
contain $\theta$ and exactly one element from each pair $\{\gamma_i,\theta-\gamma_i\}$ (cf. \cite[Prop.\,7.2]{cp3}). In particular, 

\begin{lm}    \label{lm:one-prop}
For $\ap\in\Pi_l$ and $\gamma\in\gH\setminus\{\theta\}$, we have 
$\gamma\in I(\ap)_{\min}$ if and only if $\theta-\gamma\not\in I(\ap)_{\min}$.
\end{lm}

Our next goal is to compare the upper ideals $I\langle{\curge}\mu\rangle$ and
$I(\mu)_{\max}$  ($\mu\in\Delta^+_l$). This will be achieved in two steps.

\begin{prop}    \label{prop:mu_min-inclus}
 For any $\mu\in\Delta^+_l$, we have $I(\mu)_{\min}\subset I\langle{\curge}\mu\rangle$.
\end{prop}
\begin{proof}
As above, $v_\mu\in W$ is the element of minimal length such that $v_\mu(\theta)=\mu$
and $w=v_\mu s_0$ is the minuscule element for $I(\mu)_{\min}$. Then
$I(\mu)_{\min}=\{\gamma\in\Delta^+\mid -\gamma+\delta\in \eus N(w)\}$ and
$\eus N(w)=\{\ap_0\}\cup s_0(\eus N(v_\mu))$.
Therefore, if $\gamma\in I(\mu)_{\min}$,  then either $\gamma=\theta$, or 
$-\gamma+\delta\in s_0(\eus N(v_\mu))$, i.e., $\theta-\gamma\in \eus N(v_\mu)$.
Clearly, $\eus N(v_\mu^{-1})=- v_\mu(\eus N(v_\mu))$. 
Hence $\theta-\gamma\in \eus N(v_\mu)$ if and only if 
$-v_\mu(\theta-\gamma)=v_\mu(\gamma)-\mu\in \eus N(v_\mu^{-1})$.
Consequently, 
\[
  \gamma\in I(\mu)_{\min} \ \ \& \ \ \gamma\ne\theta  \Leftrightarrow  v_\mu(\gamma)-\mu\in
  \eus N(v_\mu^{-1}) 
\]
Set $\nu=v_\mu(\gamma)-\mu$. Then $\gamma=v_\mu^{-1}(\nu+\mu)=\theta+v_\mu^{-1}(\nu)$. Hence our goal is to prove that 
\\[.8ex]
\hbox to \textwidth{\ $(\ast)$ \hfil
{\it for any $\nu\in \eus N(v_\mu^{-1})$, one has
$\theta+v_\mu^{-1}(\nu)\curge \mu$.} \hfil }

We will argue by induction on $\ell(v_\mu)=(\rho, \theta^\vee-\mu^\vee)$. To perform the induction step, assume that $\mu\not\in\Pi_l$ and $(\ast)$ is satisfied. Take any $\ap\in\Pi$
such that $(\ap,\mu)>0$ and set $\mu':=s_\ap(\mu)\prec \mu$.
Consider $v_{\mu'}=s_\ap v_\mu$, which corresponds to the minuscule element
$w'=s_\ap w=v_{\mu'} s_0$ (Proposition~\ref{prop:elem-ext}) and the larger abelian ideal 
$I(\mu')_{\min}$. Then 
\[
    \eus N(v_{\mu'}^{-1})=\{\ap\}\cup s_\ap(\eus N(v_\mu^{-1})) .
\]
Thus, to prove the analogue of $(\ast)$ for $\nu'\in \eus N(v_{\mu'}^{-1})$, we have to handle two possibilities:

a) \ $\nu'=s_\ap(\nu)$ for $\nu\in \eus N(v_\mu^{-1})$. \\
Then $\theta+v_{\mu'}^{-1}(\nu')=\theta+v_\mu^{-1}(\nu)\curge \mu \succ \mu'$, as required.

b) \ $\nu'=\ap$. \\
We have to prove here that $\theta+v_{\mu'}^{-1}(\ap)=\theta-v_\mu^{-1}(\ap)\curge \mu'=s_\ap(\mu)$.
To this end, take a reduced decomposition 
$v_\mu^{-1}=s_{\gamma_k}{\cdots} s_{\gamma_1}$, where 
$\{ \gamma_1,\ldots,\gamma_k\}$ is a multiset of simple roots. 
Recall that $v_\mu^{-1}(\mu)=\theta$  and $k=(\rho,\theta^\vee-\mu^\vee)$. 
Since $(\rho, s_\ap(\nu)^\vee-\nu^\vee)\in \{-1,0,1\}$ for any $\nu\in\Delta^+_l$ and $\ap\in\Pi$, the chain of roots 
\[
\mu_0=\mu,\ \mu_1=s_{\gamma_1}(\mu),\ \mu_2=s_{\gamma_2}s_{\gamma_1}(\mu),\dots,\ \mu_k=\theta ,
\]
has the property that $\mu_i\prec \mu_{i+1}$ and 
each simple reflection $s_{\gamma_i}$ increases the "level" $(\rho, (\cdot)^\vee)$ by $1$. Then
we must have $\theta=\mu+\sum_{i=1}^k n_i\gamma_i$, where \\ 
$n_i=\begin{cases} 1 &  \text{if $\gamma_i$ is long} \\
                              |\! |{\rm long}|\! |^2 / |\! | {\rm short} |\! |^2 &  \text{if $\gamma_i$ is short}.
                              \end{cases}$
\\
We also have $s_\ap(\mu)=\mu-(\mu,\ap^\vee)\ap$ and 
$v_{\mu}^{-1}(\ap)\curle \ap +  \sum_{i=1}^k n_i\gamma_i$.  Whence
\[
  v_\mu^{-1}(\ap)+ s_\ap(\mu) \curle \mu + \sum_{i=1}^k n_i\gamma_i +(1-(\mu,\ap^\vee))\ap \curle \theta .
\]
This completes the induction step and proof of proposition.
\end{proof}
\begin{thm}    \label{thm:mu_max-inclus}
 For any $\mu\in\Delta^+_l$, we have $I(\mu)_{\max}\subset I\langle{\curge}\mu\rangle$.
 In particular, if $I\in\Ab_\mu$, then $I\subset I\langle{\curge}\mu\rangle$.
\end{thm}
\begin{proof}
Suppose that $\gamma\in I(\mu)_{\max}$. In particular, $\gamma$ is a commutative root. 
\\
\textbullet\quad If  $\gamma\in\Delta^+_l$, then the ideal $I(\gamma)_{\min}$ is well-defined and
\[
  I(\gamma)_{\min}\ \underset{Prop.~\ref{prop:mu_min-inclus}}{\subset}\  I\langle{\curge}\gamma\rangle\cap \gH\ \ {\subset} \ \ 
  I(\mu)_{\max}\cap \gH\underset{Prop.~\ref{prop:other-in-fibre}}{=}I(\mu)_{\min} .
\]
By \cite[Thm.\,4.5]{imrn}, we conclude that $\gamma\curge\mu$. (This completes the proof in the {\bf A-D-E} case!)
\\
\textbullet\quad If $\gamma$ is short and $\gamma\in \gH$, then $\gamma\in I(\mu)_{\min}\subset
 I\langle{\curge}\mu\rangle$ by Propositions~~\ref{prop:other-in-fibre} and 
 \ref{prop:mu_min-inclus}. 
\\
\textbullet\quad  The remaining possibility is that $\gamma$ is short and $\gamma\not\in \gH$. 
But, there is no such commutative roots for $\GR{B}{n},\GR{F}{4},\GR{G}{2}$. (For 
$\GR{B}{n}$, the only short commutative root is $\esi_1$ and $\theta=\esi_1+\esi_2$.)
For $\GR{C}{n}$, such commutative roots are of the form $\gamma=\esi_i+\esi_j$
with $2\le i<j\le n$. Here $\gH=\{\esi_1\pm\esi_j\mid 2\le j\le n\}\cup\{2\esi_1\}$ and $I\langle {\curge}\esi_i+\esi_j\rangle\cap\gH=I\langle {\curge}\esi_1+\esi_j\rangle$.
Then using Proposition~\ref{prop:other-in-fibre} shows that 
$\rt \bigl(I\langle {\curge}\esi_i+\esi_j\rangle\bigr)=2\esi_j$. Clearly, we have
$\esi_i+\esi_j\curge 2\esi_j$.  
(As usual, the simple roots of $\GR{C}{n}$ are $\esi_1-\esi_2,\dots,\esi_{n-1}-\esi_n,2\esi_n$.)
\end{proof}

\begin{rmk}
If $\g$ is of type $\GR{A}{n}$ or $\GR{C}{n}$, then 
$I(\mu)_{\max}= I\langle{\curge}\mu\rangle$ for all $\mu\in\Delta^+_l$. For all other
types, this is not always the case.
\end{rmk}

\section{Centralisers of abelian ideals}
\label{sect:central} 

\noindent In this section, we mostly regard abelian ideals as subspaces $\ah$ of $\ut$.
Accordingly, for $\mu\in\Delta^+_l$,  the minimal and maximal elements
of $\Ab_\mu$ are denoted by $\ah(\mu)_{\min}$ and $\ah(\mu)_{\max}$, respectively.

If $\ce\subset \g$ is a subspace, then $\z_\g(\ce)$ denotes the {\it centraliser\/} of
$\ce$ in $\g$. If $\ce$ is $\be$-stable, then so is $\z_\g(\ce)$.
If $\ah\in\Ab$, then $\z_\g(\ah)$ is a $\be$-stable subalgebra of $\g$ and 
$\z_\g(\ah)\supset \ah$. However, $\z_\g(\ah)$ may contain semisimple elements
and/or it may happen that $\z_\g(\ah)\not\subset \be$.

Consider the following properties of abelian ideals:

\noindent
(P1): \  $\z_\g(\ah)$ belongs to $\ut$; \
(P2): \ $\z_\g(\ah)$ a sum of abelian ideals; \ 
(P3): \  $\z_\g(\ah)$ an abelian ideal.

Clearly, (P3)$\Rightarrow$(P2)$\Rightarrow$(P1).

\noindent
We say that $\ah$ is of {\it full rank}, if $I_\ah$ contains $n$ linearly independent roots ($n=\rk\g$).

\begin{lm}  \label{lm:full-rang}
Let $\ah\in\Ab$. Then $\z_\g(\ah)\subset \ut$ if and only if $\ah$ is of full rank.
\end{lm}
\begin{proof}
If $\ah$ is not of full rank, then $\z_\g(\ah)\cap\te\ne 0$. If $\ah$ is of full rank, then 
$\z_\g(\ah)\cap \te=0$ and $\z_\g(\ah)$ is $\be$-stable. Therefore, $\z_\g(\ah)$ cannot contain
root spaces corresponding to negative roots.
\end{proof}

\begin{lm}   \label{lm:sum-abelian}
$\z_\g(\ah)$ is a sum of abelian ideals if and only if\/ $\ah$ is of full rank and 
$\theta-[\theta/2]\in I_\ah$.
\end{lm}
\begin{proof} The root space $\g_{[\theta/2]}$ belongs to $\z_\g(\ah)$ if and only if 
$\theta-[\theta/2]\not\in I_\ah$. 
The rest follows from the fact that $[\theta/2]$ is the unique maximal  
noncommutative root.
\end{proof}

Recall that $\{\ah(\ap)_{\max}  \mid \ap\in \Pi_l\}$ is the complete set of 
maximal abelian ideals.
For any $\ah\in\Ab$, $\z_\g(\ah)$ contains the sum of all maximal abelian ideals that 
contain $\ah$. Therefore, if $\z_\g(\ah)$ is an abelian ideal, then 
$\z_\g(\ah)=\ah(\ap)_{\max}$ for some $\ap\in\Pi_l$ and $\ah(\ap)_{\max}$ is
the only maximal abelian ideal containing $\ah$. 

\begin{lm}                 \label{lm:3}
An abelian ideal $\ah$ belongs to a unique maximal abelian ideal if and only if
there is a unique $\ap\in\Pi_l$ such that $\rt(\ah)\curge \ap$. In particular, in the simply-laced case, the last condition means precisely that $\rt(\ah)\in \Pi_l$. 
\end{lm}
\begin{proof}
Follows from the fact that the inclusion $\ah\subset\tilde\ah$ implies that $\rt(\ah)\curge
\rt(\tilde\ah)$, see \cite[Cor.\,3.3]{imrn}. (Cf. also  \cite[Thm.\,2.6(3)]{imrn}.)
\end{proof}

Note that if $\ah$ is a maximal abelian ideal, then $\z_\g(\ah)=\ah$ and thereby 
$\z_\g(\ah)$ is an abelian ideal. For, if $\z_\g(\ah)\supsetneqq\ah$ and $\gamma$ is a 
maximal element in $I_{\z_\g(\ah)}\setminus I_\ah$, then $\ah\oplus \g_\gamma$ would be 
a larger abelian ideal!
To get a general answer, we need some preparatory results.
\begin{lm}        \label{lm:4}
For any $\ap\in\Pi_l$, the ideal $\ah(\ap)_{\min}$ is of full rank. 
\end{lm}
\begin{proof}
By \cite[Thm.\,4.3]{imrn}, the corresponding minuscule element $w\in\HW$
equals $v_\ap s_0$, where $v_\ap\in W$ is the unique element of minimal length 
taking $\theta$ to $\ap$. 
Since $v_\ap(\theta)=\ap$, any reduced decomposition of  $v_\ap$ contains all simple reflections
corresponding to $\Pi\setminus \{\ap\}$. Therefore  $w$ contains reflections corresponding
to $n=\#(\Pi)$ linearly independent roots. This easily implies that the inversion set
$\eus N(w)$ contains $n$ linearly independent affine roots. Hence $\ah(\ap)_{\min}$ is of full rank.
\end{proof}

\begin{lm}   \label{lm:max-I-max}
For any $\ap\in\Pi_l$, we have 
$\max\bigl(\Delta^+\setminus I(\ap)_{\max}\bigr)\subset \gH\setminus\{\theta\}$.
\end{lm}
\begin{proof}
If $\gamma\in \max\bigl(\Delta^+\setminus I(\ap)_{\max}\bigr)$, then $I(\ap)_{\max}\cup\{\gamma\}$ determines a $\be$-stable subspace of $\ut$, which is no longer abelian. That is, there exists
$\xi\in I(\ap)_{\max}$ such  that $\xi+\gamma\in\Delta^+$. Then there are $\xi'\curge \xi$ and 
$\gamma'\curge \gamma$ such that $\xi'+\gamma'=\theta$, see \cite[p.\,1897]{imrn}. 
Since $I(\ap)_{\max}$ is abelian, this
clearly implies that $\gamma'=\gamma$, hence $\gamma\in  \gH\setminus\{\theta\}$.
\end{proof}

\begin{thm}     
Let $\ah\in\Ab$. The following conditions are equivalent:
\begin{itemize}
\item[\sf (1)] \  $\z_\g(\ah)$ is an abelian ideal;
\item[\sf (2)] \  $\z_\g(\ah)=\ah(\ap)_{\max}$ \ for some $\ap\in\Pi_l$;
\item[\sf (3)] \  $\rt(\ah)\in\Pi_l$.
\end{itemize}
\end{thm}
\begin{proof}
(1)$\Rightarrow$(2): \ See the paragraph in front of Lemma~\ref{lm:3}.

(2)$\Rightarrow$(1): \ Obvious.  

(2)$\Rightarrow$(3): \ Here $\ah(\ap)_{\max}$ is the only maximal abelian ideal that contains $\ah$. Therefore, in the simply-laced case, the assertion follows from Lemma~\ref{lm:3}. 

For the non-simply-laced case, assume that $\rt(\ah)=\gamma\not\in\Pi_l$, but still
$\gamma$ majorizes a unique long simple root. Then $\gamma$ also majorizes  a short simple 
root, whence $\gamma\not\curle |\Pi_l|$.
We claim that $\theta-[\theta/2]\not\in I_\ah$, and thereby $\z_\g(\ah)$ is not a sum of abelian ideals, in view of Lemma~\ref{lm:sum-abelian}.
Indeed, assume that $\theta-[\theta/2]\in I_\ah$. Then $I_\ah$ contains the upper ideal of
$\Delta^+$ generated by $\theta-[\theta/2]$, which is exactly
$\bigcap_{\ap\in\Pi_l}I(\ap)_{\min}=I(|\Pi_l|)_{\min}$, see Example~\ref{ex:all-maximal}.
Then the inclusion $I_\ah\supset I(|\Pi_l|)_{\min}$ implies that 
$\gamma=\rt(\ah)\curle |\Pi_l|$, a contradiction!

(3)$\Rightarrow$(2): \ It suffices to prove that the centraliser of $\ah(\ap)_{\min}$
equals $\ah(\ap)_{\max}$ for any $\ap\in\Pi_l$.
To this end, we have to check that:

(i) \ $\z_\g(\ah(\ap)_{\min})$ contains no semisimple 
elements of $\g$ (i.e., $\ah(\ap)_{\min}$ is of full rank), and 

(ii) \ the nilpotent subalgebra $\z_\g(\ah(\ap)_{\min})$ cannot be larger
than $\ah(\ap)_{\max}$, i.e., for any 
$\gamma\in \max \bigl(\Delta^+\setminus I(\ap)_{\max}\bigr)$, 
there exists a $\nu\in I(\ap)_{\min}$ such that $\gamma+\nu\in\Delta^+$.

For (i):  This is Lemma~\ref{lm:4}.

For (ii):  If $\gamma\in \max \bigl(\Delta^+\setminus I(\ap)_{\max}\bigr)$, then $\gamma\in
\gH\setminus\{\theta\}$ (Lemma~\ref{lm:max-I-max}). Then $\theta-\gamma\in I(\ap)_{\min}$ in view
of Lemma~\ref{lm:one-prop}.
\end{proof}

\begin{thm}  \label{thm:stunning}
For $\ap\in\Pi_l$, there is a one-to-one correspondence between
$\min\bigl( I(\ap)_{\min}\bigr)$ and $\max \bigl(\Delta^+\setminus I(\ap)_{\max}\bigr)$.
Namely, for every $\eta\in 
\max \bigl(\Delta^+\setminus I(\ap)_{\max}\bigr)$, there is $\eta' \in \min\bigl( I(\ap)_{\min}\bigr)$ 
such that $\eta+\eta'=\theta$; 
and vice versa.
\end{thm}
\begin{proof}
We assume that $\rk\Delta >1$, hence $\theta\ne \ap$, $I(\ap)_{\min}\ne\{\theta\}$,
and $\gH\setminus\{\theta\}\ne\varnothing$.

1) \ If $\eta\in  \max \bigl(\Delta^+\setminus I(\ap)_{\max}\bigr)$, then 
$\eta\in \gH\setminus\{\theta\}$ (Lemma~\ref{lm:max-I-max}) and also $\eta\not\in I(\ap)_{\min}$.
Hence
$\eta'=\theta-\eta\in I(\ap)_{\min}$ (Lemma~\ref{lm:one-prop}). Actually,
$\eta'$ is a minimal element of $I(\ap)_{\min}$. For, if $\xi'\in I(\ap)_{\min}$ and 
$\eta'\succ \xi'$, then $\theta-\xi'\succ\eta$ and hence $\theta-\xi'\in I(\ap)_{\max}$.
As $I(\ap)_{\max}\cap\gH=I(\ap)_{\min}$, one would obtain $\theta-\xi'\in I(\ap)_{\min}$, which contradicts Lemma~\ref{lm:one-prop}.

2)  If $\eta'\in I(\ap)_{\min}$, then $\eta:=\theta-\eta'\in (\gH\setminus I(\ap)_{\min})$. Hence
$\eta\not\in I(\ap)_{\max}$. Assume that $\eta$ is not maximal in $\Delta^+\setminus I(\ap)_{\max}$
and $\xi\succ\eta$ with $\xi\not\in I(\ap)_{\max}$. Then $\theta-\xi\prec \eta'$ and $\theta-\xi\in
I(\ap)_{\min}$, which contradicts the choice of $\eta'$.
\end{proof}
\begin{rema}
Lemma~\ref{lm:max-I-max} and the above proof of Theorem~\ref{thm:stunning} 
(=\,Theorem~\ref{thm:intr4}) are based on the suggestion of the anonymous referee. This uniform 
proof replaces our initial case-by-case considerations.
\end{rema}
\begin{ex}
We describe the corresponding minimal and maximal elements in the two extreme cases---the most classical ($\GR{A}{n}$) and
most exceptional ($\GR{E}{8}$).

As usual,  $\Delta^+(\GR{A}{n})=\{ \esi_i-\esi_j \mid1\le i<j\le n+1\}$, and 
$\ap_i=\esi_i-\esi_{i+1}$. Here \\
$\min\bigl( I(\ap_i)_{\max}\bigr){=}\{\ap_i\}$ and $\gH=\{\esi_i-\esi_j \mid i=1 \text{ or } j=n+1\}$. Therefore
 
$\max\bigl(\Delta^+\setminus I(\ap_i)_{\max}\bigr)=\{\esi_1{-}\esi_i, \esi_{i+1}{-}\esi_{n+1} \}$, 
$\min\bigl( I(\ap_i)_{\min}\bigr)=\{\esi_i{-}\esi_{n+1}, \esi_1{-}\esi_{i+1} \}$.
\\
The respective roots in the previous row sums to $\theta=\esi_1-\esi_{n+1}$.

For $\GR{E}{8}$, we use the natural numbering of $\Pi$, i.e.,
$\left(\text{\begin{E8}{1}{2}{3}{4}{5}{6}{7}{8}\end{E8}}\right)$. The root 
$\gamma=\sum_{i=1}^8 n_i\ap_i$ is denoted by $n_1 n_2\dots n_8$.
Here $\theta=23456423$ and $\gamma\in\gH$ if and only if $n_1\ne 0$. The respective maximal and minimal elements are gathered in Table~\ref{tabl-E}.

\begin{table}[htb]   
\begin{center}
\caption{Data for the root system of type $\GR{E}{8}$}
\begin{tabular}{cccc}  \label{tabl-E}
$i$  &   $\min\bigl( I(\ap_i)_{\max}\bigr)$ & $\min\bigl( I(\ap_i)_{\min}\bigr)$ & 
$\max\bigl(\Delta^+\setminus I(\ap_i)_{\max}\bigr)$ \\ \hline 
1 & 12222101 & 12222101 &  11234322 \\ \cline{2-4}
2 & 12222111 & 12222111  &  11234312 \\
   & 01234322 & 11234322 &   12222101 \\ \cline{2-4}
3 & 12222211 & 12222211 &  11234212  \\
   & 01234312 & 11234312 &  12222111 \\ \cline{2-4}
4 & 12223211 & 12223211 &  11233212 \\
   & 01234212 & 11234212 &  12222211 \\ \cline{2-4}
   & 12223212 & 12223212 &  11233211 \\
5 & 12233211 & 12233211  &  11223212 \\
   & 01233212 & 11233212  &  12223211 \\ \cline{2-4}
6 & 12333211 &  12333211 &  11123212 \\
   & 01223212 & 11223212  &  12233211 \\ \cline{2-4}
7 & 00123212 & 11123212  &  12333211 \\ \cline{2-4}
8 & 01233211 & 11233211  &  12223212 \\  \hline
\end{tabular}
\end{center}
\end{table}
\end{ex}

\noindent
Theorem~\ref{thm:stunning} is not true for arbitrary  long roots in place
of $\ap\in\Pi_l$. However, it can  be extended as follows. 

\begin{thm}       \label{thm:modification}
Let $\gS$ be any connected subset of $\Pi_l$. Then there is 
a one-to-one correspondence between
$
\min\bigl(\bigcap_{\ap_i\in\gS} I(\ap_i)_{\min}\bigr)$ and 
$
\max\bigl(\bigcap_{\ap_i\in\gS}(\Delta^+\setminus I(\ap_i)_{\max})\bigr)$.
Namely, for every
$\nu \in \min\bigl(\bigcap_{\ap_i\in\gS} I(\ap_i)_{\min}\bigr)$, there is 
$\nu'\in \max \bigl(\bigcap_{\ap_i\in\gS}(\Delta^+\setminus I(\ap_i)_{\max})\bigr)$ 
such that $\nu+\nu'=\theta$.
\end{thm}

Our proof is based on direct calculations, which are omitted.
It's would be interesting to find a conceptual argument.

\begin{ex}
For $\# \gS=1$, we have Theorem~\ref{thm:stunning}. At the other extreme, if $\gS=\Pi_l$, then $\bigcap_{\ap_i\in\Pi_l}(\Delta^+\setminus I(\ap_i)_{\max})=
\Delta^+\setminus  \bigcup_{\ap_i\in\Pi_l}I(\ap_i)_{\max}=
\Delta^+\setminus \Delta^+_{\mathsf{com}}$. 
Therefore, 
\[
   \max\Bigl(\bigcap_{\ap_i\in\Pi_l}\bigl(\Delta^+\setminus I(\ap_i)_{\max}\bigr)\Bigr)=\{[\theta/2]\} .
\]
Also, $\bigcap_{\ap_i\in\Pi_l} I(\ap_i)_{\min}=I(|\Pi_l|)_{\min}$ and the unique minimal element of this ideal is $\theta-[\theta/2]$, see Example~\ref{ex:all-maximal}.
Thus, an a priori proof of Theorem~\ref{thm:modification} would provide an explanation
of properties of abelian ideals with rootlet $|\Pi_l|$, cf. Example~\ref{ex:all-maximal} and
Remark~\ref{rmk:comm-roots}.
\end{ex}

\appendix
\section{A property of root systems} 
\label{app:A}

\noindent
Let $\Delta$ be a reduced irreducible root system, with a set of simple roots
$\Pi=\{\ap_1,\dots,\ap_n\}$. 

\begin{df}   \label{def:join}
Let $\eta,\beta\in\Delta^+$.
The root $\kappa$ is  the {\it least upper bound} (or {\it join\/}) of $\eta$ and $\beta$, if 
\begin{itemize}
\item \ $\kappa\curge \eta,\ \kappa\curge \beta$;
\item \ if  $\kappa'\curge \eta,\ \kappa'\curge \beta$, then $\kappa'\curge \kappa$.
\end{itemize}
The join of $\eta$ and $\beta$  is denoted by $\eta\vee\beta$.
\end{df}

\noindent
Our goal is to prove that $\eta\vee\beta$ exists for all pairs $(\eta,\beta)$, i.e., 
$(\Delta^+, \curle)$ is a join-semilattice (see \cite[3.3]{stan} about lattices). 
We actually prove a more precise assertion.
For any pair $\eta,\beta\in \Delta^+$, we define an element $\eta\vee\beta\in Q$ and then prove that it is always a  root.
The very construction of $\eta\vee\beta$ will make it clear that this root satisfy the conditions
of Definition~\ref{def:join}. We also prove that $\eta\vee\beta\in\Delta_l$, whenever 
$\eta,\beta\in\Delta_l$, so that this general setup is compatible with that of 
Section~\ref{sect:intersect}. This goes as follows.
If $\eta=\sum_{i=1}^n a_i\ap_i$, then $\hot(\eta)=\sum_ia_i$ and the {\it support\/} of 
$\eta$ is $\supp(\eta)=\{\ap_i\mid a_i\ne 0\}$. We regard $\supp(\eta)$ as subset of the 
Dynkin diagram $\eus D(\Delta)$. 
As is well known, $\supp(\eta)$ is a connected subset of $\eus D(\Delta)$
for all $\eta\in \Delta$ \cite[Ch. VI, \S\,1, n.6]{bour}.
If $\beta=\sum_{i=1}^n b_i\ap_i$, then 
$\max\{\eta,\beta\}:=\sum_{i=1}^n \max\{a_i,b_i\}\ap_i$. In general, it is merely an element of $Q$.

Say that $\supp(\eta)$ and $\supp(\beta)$ are {\it disjoint}, if 
$\supp(\eta)\cup\supp(\beta)$ is disconnected. Then there is a unique chain in
$\eus D(\Delta)$ connecting both supports, since $\eus D(\Delta)$ is a tree. 
If this chain consists of simple roots $\{\ap_{i_1},\dots,\ap_{i_s}\}$, then, by definition,
the {\it connecting root\/} is  $\ap_{i_1}+\ldots +\ap_{i_s}$. By \cite[Ch. VI, \S\,1, n.6, Cor.\,3]{bour},
it is indeed a root.

\begin{thm}   \label{thm-app2}
{$\mathsf 1^o$.} \ If\/ $\supp(\eta)\cup \supp(\beta)$ is a connected subset of\/ $\eus D(\Delta)$, then
$\eta\vee\beta=\max\{\eta,\beta\}$.

{$\mathsf 2^o$.} \ If\/ $\supp(\eta)$ and $\supp(\beta)$ are disjoint, then 
$\eta\vee\beta=\eta+\beta+(\text{connecting root})$.
\end{thm}
\begin{proof}
1) Obviously, if $\kappa\curge \eta,\ \kappa\curge \beta$, then $\kappa\curge
\max\{\eta,\beta\}$. Hence it suffices to prove that here $\max\{\eta,\beta\}$ is a root.

\textbullet \quad If $\supp(\eta)\cap \supp(\beta)=\varnothing$, then $\max\{\eta,\beta\}=\eta+\beta$. 
Since $\supp(\eta)\cup \supp(\beta)$ is connected, we have $(\eta,\beta)<0$. Hence $\eta+\beta$
is a root, and we are done.

\textbullet \quad Assume that 
$\supp(\eta)\cap \supp(\beta)\ne\varnothing$.
Without loss of generality, we may also assume that  $\hot(\eta)\ge \hot(\beta)$. 
Then we will 
argue by induction on $\hot(\beta)$.
\\[.6ex]
-- \ If $\hot(\beta)=1$, then $\beta\in \supp(\eta)$ and $\max\{\eta,\beta\}=\eta$.
\\
-- \ Suppose that $\hot(\beta)>1$ and the assertion is true for all pairs of positive 
roots such that one of them has height strictly less than $\hot(\beta)$.

Assume that there are different simple roots $\ap',\ap''$ such that $\beta-\ap',
\beta-\ap''\in \Delta^+$. Then $\max\{\beta-\ap',\beta-\ap''\}=\beta$, and
by the induction assumption 
\[
 \max\{\eta,\beta\}=\max\bigl\{\max\{\eta,\beta-\ap'\}, \beta-\ap''\bigr\} \in \Delta^+ .
 \]
It remains to handle the case in which there is a unique $\ap\in\Pi$ such that
$\beta-\ap\in \Delta^+$. Let $\hot_\ap(\beta)$ denote the coefficient of $\ap$ 
in the expression of $\beta$ via the simple roots. 
Set $\Delta_\ap(i)=\{\nu\in\Delta^+\mid  \hot_\ap(\nu)=i\}$.
By a result of Kostant (see Joseph's exposition in \cite[2.1]{jos98}), each $\Delta_\ap(i)$ 
is the set of weights of a \un{sim}p\un{le} $\el$-module, where $\el$ is the Levi subalgebra of $\g$ 
whose set of simple roots is $\Pi\setminus\{\ap\}$. Therefore, $\Delta_\ap(i)$
has a unique minimal and unique maximal
elements. Clearly, $\beta$ is the minimal element in $\Delta_\ap(j)$, where
$\hot_\ap(\beta)=j$. This also implies that 
if $\hot_\ap(\nu)\ge j$, then $\nu\curge \beta$. Therefore, if $\hot_\ap(\eta)\ge j$, 
then $\max\{\eta,\beta\}=\eta$. Hence we may assume that $\hot_\ap(\eta)\le j-1$.
Since $\supp(\eta)\cup \supp(\beta)$ is connected and 
$\supp(\eta)\cap \supp(\beta)\ne\varnothing$, the union 
$\supp(\eta)\cup \supp(\beta-\ap)$ is still connected.
Therefore 
\[
\max\{\eta,\beta-\ap\}\in \Delta^+ \ \ \text{ and } \  \ 
\hot_\ap(\max\{\eta,\beta-\ap\})=j-1 .
\] 
Hence $\max\{\eta,\beta-\ap\}+\ap=
\max\{\eta,\beta\}$ and our task is to prove that, under these circumstances, 
$\max\{\eta,\beta-\ap\}+\ap$ is a root.

Since $\hot_\ap(\max\{\eta,\beta-\ap\})=\hot_\ap(\beta-\ap)$, we  have
$(\max\{\eta,\beta-\ap\}, \ap)\le (\beta-\ap,\ap)$.
If $|\!|\ap |\!| \ge |\!|\beta |\!|$, then $\beta-\ap=s_\ap(\beta)$ and
$(\beta-\ap,\ap)<0$. This implies that 
$\max\{\eta,\beta-\ap\}+\ap\in\Delta^+$ 
\quad 
(and completes the proof of part 1$^o$, if all the roots
have the same length!)

Suppose that $|\!|\ap |\!| < |\!|\beta |\!|$. 
We exclude the obvious case when $\Delta$ is of type $\GR{G}{2}$ and assume that
$|\!|\beta |\!|/|\!|\ap |\!|=\sqrt{2}$.
Then $s_\ap(\beta)=\beta-2\ap$, $\beta-\ap$ is short, and $(\beta-\ap, \ap)=0$.

Now, if $(\max\{\eta,\beta-\ap\}, \ap) < (\beta-\ap,\ap)$, we  again 
conclude that  $\max\{\eta,\beta-\ap\}+\ap\in\Delta^+$. 
The other possibility is that $(\max\{\eta,\beta-\ap\}, \ap)=(\beta-\ap,\ap)=0$. 
Because $\hot_\ap(\max\{\eta,\beta-\ap\})=\hot_\ap(\beta-\ap)$, this means that 
$\max\{\eta,\beta-\ap\}$ and $\beta-\ap$ have also the same coefficients on the simple roots adjacent to $\ap$. That is, $\max\{\eta,\beta-\ap\}$
is obtained from $\beta-\ap$ by adding a sequence of simple roots that are {\sl orthogonal\/} to $\ap$.
Therefore, arguing by induction on $\hot(\max\{\eta,\beta-\ap\})-\hot(\beta-\ap)$,
we are left with the following problem: 

{\it Suppose that $\ap,\ap'$ are 
orthogonal simple roots such that $\nu,\nu+\ap, \nu+\ap'\in \Delta^+$, 
both $\nu$ and $\ap$ are short, and $(\nu,\ap)=0$.
Prove that $\nu+\ap+\ap'\in \Delta^+$.  
} \\
Now, if $(\ap',\nu)<0$, then $(\ap',\nu-\ap)<0$ as well.
Hence $\nu-\ap+\ap'\in\Delta$ and $s_{\ap}(\nu-\ap+\ap')=
\nu+\ap+\ap'\in \Delta^+$, as required.
The remaining conceivable possibility is that $\nu,\ap,\ap'$ are pairwise orthogonal and 
short. A quick case-by-case argument shows that this
is actually impossible.

2) In this case, at least one support, say $\supp(\beta)$, 
is a chain with all roots of the same length. Therefore, $\beta$ equals the sum of all simple roots 
in its support. Hence $\tilde\beta=\beta+(\text{connecting root})$ is a root.
Then $\supp(\eta)\cap\supp(\tilde\beta)=\varnothing$ and $\supp(\eta)\cup\supp(\tilde\beta)$
is connected. Hence $(\eta,\tilde\beta)<0$ and $\eta+\tilde\beta\in\Delta^+$.

Obviously, $\eta+\tilde\beta$ is the minimal root  that majorizes  both  $\eta$ and $\beta$.
\end{proof}

An equivalent formulation of Theorem~\ref{thm-app2} is:
\\
{\it The intersection of two principal upper ideals in $\Delta^+$ is again a principal ideal.
That is, $I\langle{\curge}\eta\rangle \cap I\langle{\curge}\beta\rangle=
I\langle{\curge}(\eta\vee\beta)\rangle$.}

\begin{cl}
 If $\eta,\beta\in\Delta^+_l$, then $\eta\vee\beta$ is also long.
\end{cl}
\begin{proof}
Let $r$ be the squared ratio of lengths of long and short roots. (Hence $r\in\{1,2,3\}$.) 
Then $\eta=\sum_{i=1}^na_i \ap_i$ is long if and only if $r$ divides $n_i$ whenever 
$\ap_i$ is short, see \cite[Ch.\,VI, \S\,1, Ex.\,20]{bour}.
Obviously, the formulae of Theorem~\ref{thm-app2} preserve this property.
\end{proof}

Recall that $\Delta_\ap(i)=\{\nu\in\Delta^+\mid  \hot_\ap(\nu)=i\}$ if 
$\ap\in\Pi$.  We regard it as a  subposet of $\Delta^+$.

\begin{cl}  \label{cor:2}
For any $\ap\in\Pi$ and $i\in\BN$, the poset $\Delta_\ap(i)$ is a lattice. 
\end{cl}
\begin{proof}
Formulae of Theorem~\ref{thm-app2} imply that if $\eta,\beta\in \Delta_\ap(i)$,
then $\eta\vee\beta\in \Delta_\ap(i)$.
Therefore $\Delta_\ap(i)$ is a finite join-semilattice having a unique minimal element.
(The latter is a part of Kostant's result referred to above.) Hence $\Delta_\ap(i)$ is a lattice
by \cite[Prop.\,3.3.1]{stan}.
\end{proof}

\end{document}